\documentclass[11pt, letterpaper,reqno]{amsart}

\usepackage{amsmath,amsthm,amsrefs,amssymb,hyperref}
\usepackage[usenames,dvipsnames]{xcolor}
\usepackage{graphicx}
\usepackage[percent]{overpic}
\usepackage{todonotes}

\numberwithin{equation}{section}
\newtheorem{question}{Question}[section]

\newtheorem{lemma}{Lemma}[section]
\newtheorem{theorem}{Theorem}[section]

\theoremstyle{definition}

\DeclareMathOperator{\dist}{dist}

\begin{document}
\title{ The range of Hardy number on comb domains}

\author{Christina Karafyllia}  
\address{Institute for Mathematical Sciences, Stony Brook University, Stony Brook, NY 11794, U.S.A.}
\email{christina.karafyllia@stonybrook.edu}   



\begin{abstract} Let $D\ne \mathbb{C}$ be a simply connected domain and $f$ be the Riemann mapping from $\mathbb{D}$ onto $D$. The Hardy number of $D$ is the supremum of all $p$ for which $f$ belongs in the Hardy space ${H^p}\left( \mathbb{D} \right)$. A comb domain is the entire plane minus an infinite number of vertical rays symmetric with respect to the real axis. In this paper we prove that for any $p\in [1,+\infty]$, there is a comb domain with Hardy number equal to $p$ and this result is sharp. It is known that the Hardy number is related with the moments of the exit time of Brownian motion from the domain. In particular, our result implies that given $ p < q$  there exists a comb domain with finite $p$-th moment but infinite $q$-th moment if and only if $q\geq 1/2$. This answers a question posed by Boudabra and Markowsky.
\end{abstract}

\maketitle

\section{Introduction}\label{int}

The Hardy space with exponent $p>0$ \cite[pp.\ 1--2]{Dur} is denoted by ${H^p}\left( \mathbb{D} \right)$ and is defined to be the set of all holomorphic functions, $f$, on the unit disk $\mathbb{D}$ that satisfy the condition 
\[\mathop {\sup }\limits_{0 < r < 1} \int_0^{2\pi } {{{| {f( {r{e^{i\theta }}} )} |}^p}d\theta  <  + \infty } .\] 
In \cite{Han1} Hansen studied the problem of determining the numbers $p$ for which $f \in {H^p}\left( \mathbb{D} \right)$ by studying the geometry of $f\left( \mathbb{D} \right)$. For this purpose he introduced a number which he called the Hardy number of a region. In case $D\ne \mathbb{C}$ is a simply connected domain and $f$ is a Riemann mapping from $\mathbb{D}$ onto $D$, the Hardy number of $D$ is defined by 
\[{\rm h}\left( D \right) = \sup \left\{ {p > 0:f \in {H^p}\left( \mathbb{D} \right)} \right\}.\]
We note that this definition is independent of the choice of the Riemann mapping onto $D$. It is known that every conformal mapping on $\mathbb{D}$ belongs to ${H^p}\left( \mathbb{D} \right)$ for all $p \in (0,1/2)$ \cite[p.\ 50]{Dur}. Therefore, ${\rm h}\left( D \right)$ lies in $[1/2, +\infty]$.

Some basic properties of the Hardy number are that it is invariant under mappings of the form $az+b$, where $a,b\in \mathbb{C}$ and $a\ne 0$ and that if $D_1 \subset D_2$ then ${\rm h}(D_1) \ge {\rm h} (D_2)$ (see \cite[pp. 236--237]{Han1}). Moreover, we know that the Hardy number of bounded domains is equal to infinity and we can find its exact value in some specific cases of unbounded domains like strips and sectors, starlike \cite{Han1} and spiral-like regions \cite{Han2}. However, there is no general method for computing the Hardy number of a domain and thus it is a difficult problem to find its exact value. There are only some ways to estimate it for certain types of domains. See, for example, \cite{Han1,Kara,Kar,Kim,Co2} and references therein. 

Comb domains are defined in the following way. Let $\left\{ {x_n } \right\}_{n \in \mathbb{Z}} $ be a strictly increasing sequence of real numbers such that $x_0=0$ and 
\[ \mathop {\inf }\limits_{n \in \mathbb{Z}} (x_{n}  - x_{n-1} )>0.\]
Also, let $\left\{ {b_n } \right\}_{n \in \mathbb{Z}} $ be an associated sequence of positive numbers. A domain of the form 
\[C=\mathbb{C}\backslash \bigcup\limits_{n \in \mathbb{Z}} {\left\{ {x_n  + iy: |y| \ge b_n} \right\}}\] 
is called a comb domain (see Fig.\ \ref{comb}). Comb domains is an interesting class of simply connected domains and thus they have been studied from various points of view. See, for example, \cite{Bou,Kara} and references therein.

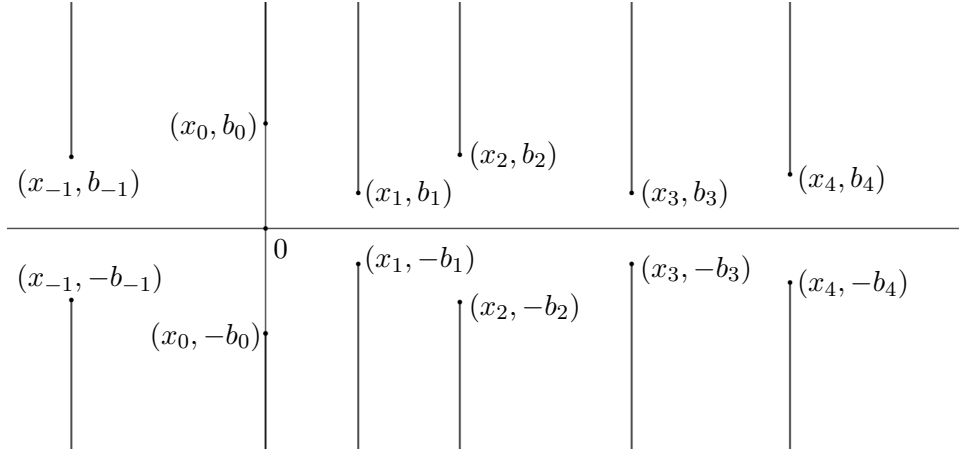
\begin{figure}
	\begin{overpic}[width=\linewidth]{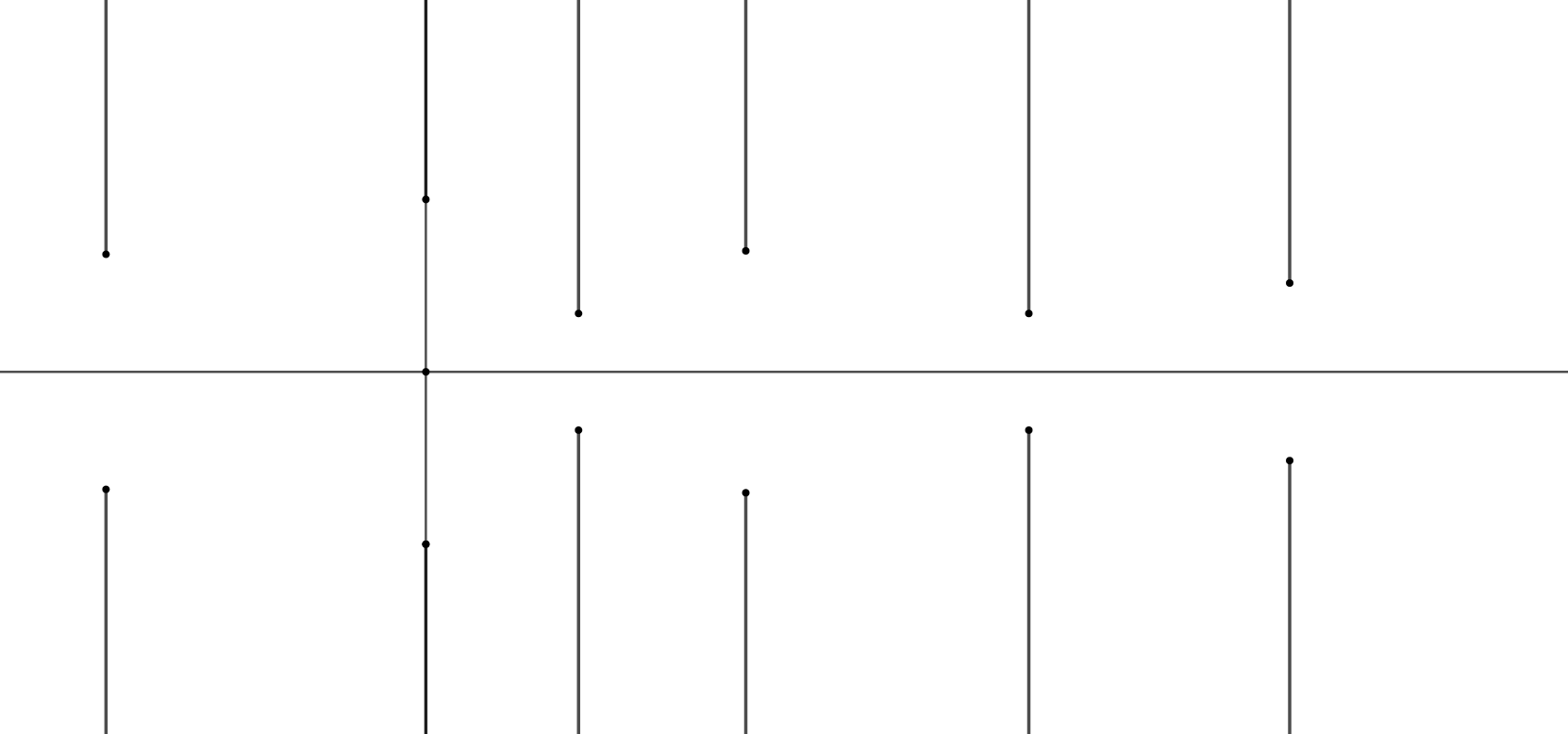}
		\put (28,20) {$0$}
		\put (17,33) {$(x_0,b_0)$}
		\put (15,11) {$(x_0,-b_0)$}
		\put (37.5,26) {$(x_1,b_1)$}
		\put (37.5,19) {$(x_1,-b_1)$}
		\put (48.5,30) {$(x_2,b_2)$}
		\put (66.5,26) {$(x_3,b_3)$}
		\put (83,27.5) {$(x_4,b_4)$}
		\put (83,16.5) {$(x_4,-b_4)$}
		\put (66.5,18) {$(x_3,-b_3)$}
		\put (48.5,14) {$(x_2,-b_2)$}
		\put (1,27) {$(x_{-1},b_{-1})$}
		\put (1,17) {$(x_{-1},-b_{-1})$}		
	\end{overpic}
	\caption{Comb domain.}
	\label{comb}	
\end{figure}

In \cite{Bou} Boudabra and Markowsky studied the finiteness of moments of the exit time of planar Brownian motion from comb domains and posed the following question.

\begin{question}\label{que1}
Given $p < q$, can we construct a comb domain with finite $p$-th moment but infinite $q$-th moment? 
\end{question}
 
The moment of the exit time of Brownian motion from a domain $D$ is related with the Hardy number of $D$. In \cite{Bur} Burkholder proved that finiteness of the $p$-th Hardy norm of  a domain $D$ is equivalent to finiteness of the $p/2$-th moment of the exit time of a planar Brownian motion from $D$. More precisely, if $D$ is a simply connected domain, then we define the number ${\widetilde h}(D)$ to be the supremum of all $p>0$ for which the $p$-th moment of the exit time of Brownian motion is finite. Then Burkholder proved in \cite{Bur} that 
\begin{equation}\nonumber
{\widetilde h}(D)={\rm h}(D)/2.
\end{equation}
Therefore, in order to answer Question \ref{que1} it suffices to answer the following stronger question.

\begin{question}
Given $p\ge 1/2$, can we construct a comb domain with Hardy number equal to $p$? 
\end{question}

In this paper we answer this question. First, we observe that the Hardy number of comb domains is always greater than or equal to 1. Therefore, if $p$ lies in $[1/2,1)$, there is no comb domain with Hardy number equal to $p$. However, if $p$ is any number in $[1,+\infty]$, we can construct a comb domain with Hardy number equal to $p$. 

\begin{theorem}\label{main}
If $C$ is a comb domain, then ${\rm h}(C)\geq 1$. Conversely, for any $p \in [1,+\infty]$, there is a comb domain $C$ with ${\rm h}(C)=p$. 
\end{theorem}
The idea is first to consider a sector with angular opening $\theta \in (0,\pi)$. We know that its Hardy number is equal to $\pi/\theta \in (1,\infty)$ \cite[p. 237]{Han1}. Then we construct a comb domain that resembles the sector and has Hardy number equal to the Hardy number of this sector. In this way we cover the cases in which $p\in (1,+\infty)$. The cases $p=1$ and $p=+\infty$ are studied separately.

\section{Preliminary results}\label{pre}

Before we prove Theorem \ref{main}, we state some results that we are going to use.

\subsection{Hardy number of starlike domains}

In \cite[p. 237]{Han1} Hansen proved that we can determine the Hardy number of a starlike domain in the following way. Let $S \ne \mathbb{C}$ be an unbounded starlike domain with respect to $0$. For $t>0$ we define
\[\alpha_S (t)=\max \{m(E):E\,\,\, \text{is} \,\,\, \text{a}\,\,\, \text{subarc}\,\,\, \text{of}\,\,\, S\cap \{|z|=t\} \},\]
where $m(E)$ denotes the angular Lebesgue measure of $E$. Then
\begin{equation}\label{starlike}
{\rm h}(S)=\lim_{t \to  + \infty } \frac{\pi}{\alpha_S (t)}.
\end{equation}

\subsection{Hardy number and hyperbolic distance}

In \cite{Kar} the current author proves that the Hardy number of a simply connected domain can be computed with the aid of the hyperbolic distance in the following way.

\begin{theorem}\label{hardy} Let $D$ be a simply connected domain with $0 \in D$. If ${F_r } =D\cap \{|z|=r\}$ for $r >0$, then
	\[{\rm h}\left( D \right)=\mathop {\liminf}\limits_{r \to  + \infty } \frac{{{d_D}\left( {0,F_r} \right)}}{{\log r }}.\]
\end{theorem}
Note that ${d_D}\left( {0,F_r} \right)$ denotes the hyperbolic distance between $0$ and $F_r$ in $D$. For the definition and properties of the hyperbolic distance see \cite{Bea}.

\subsection{Harmonic measure and hyperbolic distance}




Let $D \ne \mathbb{C}$ be a simply connected domain with $0\in D$. Let $E$ be a connected and relatively closed subset of $D$ that is not compact and $0\notin E$. 
Using the Beurling--Nevanlinna projection theorem \cite[p. 105]{Gar}, Poggi-Corradini proved in \cite[pp. 9--10]{Co1} that
\begin{equation}\label{hahy}
\omega _D\left( {0,E} \right) \ge \frac{2}{\pi }{e^{ - {d_D}\left( {0,E} \right)}.}
\end{equation}
Here $\omega _D\left( {0,E} \right)$ denotes the harmonic measure of $E$ at $0$ with respect to the component of $D\backslash E $ that contains $0$. For the definition and more properties of the harmonic measure see \cite{Gar}.

\subsection{An estimate for harmonic measure}

Next, we state an upper estimate for the harmonic measure. For the proof see \cite[pp. 369--376]{Beu} and \cite[pp. 147--149]{Gar}.

\begin{theorem}\label{upperest}
Let $D$ be a simply connected domain with $0\in D$. Also, let $E\subset {D} \cap \{|z|\ge R\}$ be a finite union of arcs. Suppose that $r_0=\dist (0,\partial D) <R$ and $J_r \subset D\cap \{|z|=r\}$ is a set that separates $0$ from $E$ for every $r\in(r_0,R)$. If $r\Theta (r)$ is the length of $J_r$, then
\[\omega _D\left( {0,E} \right) \le \frac{8}{\pi} \exp{\left(-\pi  \int_{r_0}^{R}{\frac{dr}{r\Theta (r)}}\right)}\]
provided that $\Theta (r)$ is measurable.
\end{theorem}

\section{Proof of Theorem \ref{main}}

First, we prove that if $C$ is a comb domain, then ${\rm h}(C)\geq 1$. We consider the starlike domain 
\[S=\mathbb{C}\backslash \{iy:|y|\ge b_0\}.\]
Since $\alpha_S (t)=\pi$ for every $t>b_0$, by (\ref{starlike}) we derive that ${\rm h}(S)=1$. Also, $C\subset S$ and hence ${\rm h}(C) \ge {\rm h}(S)$, that is, ${\rm h}(C) \ge 1$. 

Next, we prove that for any $p \in [1,+\infty]$, there is a comb domain with Hardy number equal to $p$. We consider three cases.

\vspace{1em}

\noindent
{\bf Case 1:} Suppose that $p\in (1,+\infty)$. Then there is an angle $\theta \in (0,\pi)$ such that $p=\pi/ \theta$. Let $S_\theta$ be the sector domain
\[\{re^{i\phi}:r>0,|\phi|<\theta/2\}\]
and $C$ be the comb domain (see Fig. \ref{cs}) for which $x_n=n$ for every $n\in \mathbb{Z}$ and 
\[b_n  = \bigg\{ \begin{array}{l}
1,\,\,n = 0 \\ 
|n|\tan (\theta /2),\,\,n \in \mathbb{Z} \backslash \{0\} 
\end{array}. \]

\begin{figure}
	\begin{overpic}[width=\linewidth]{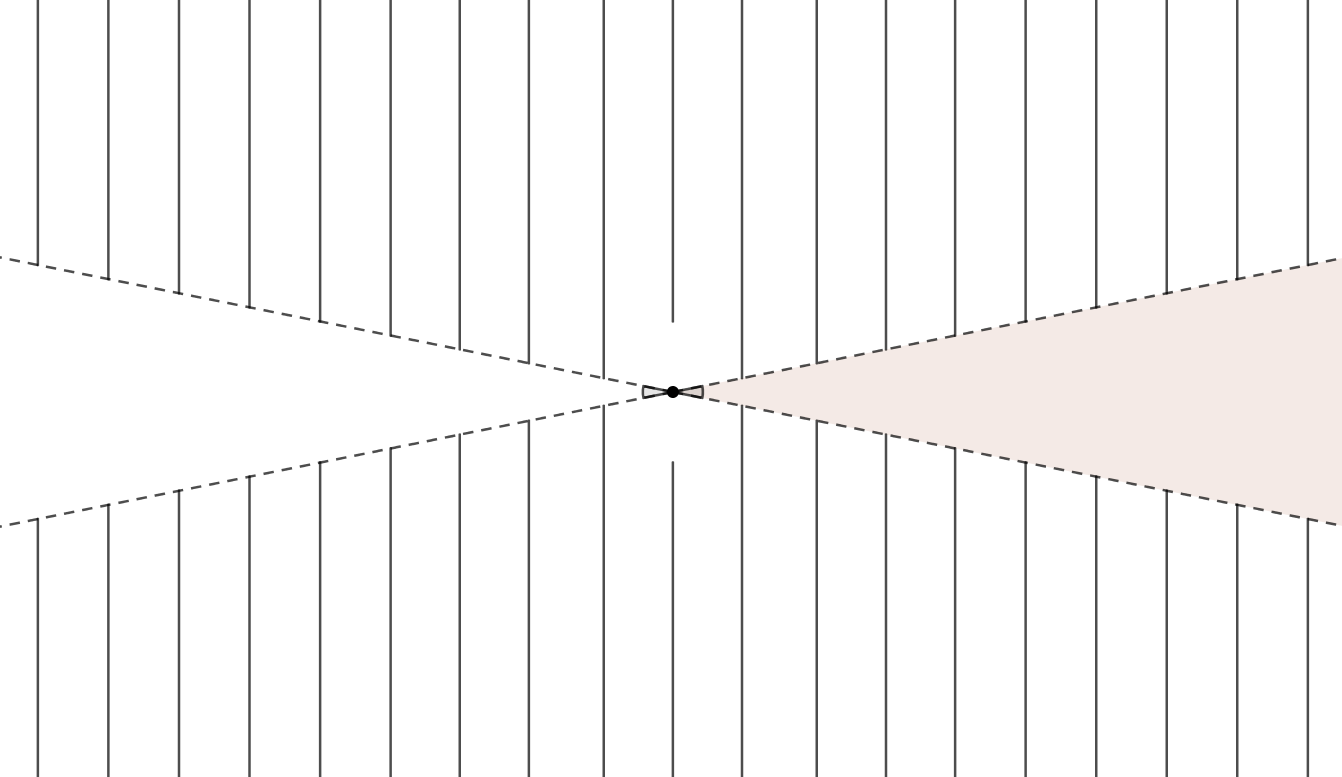}
		\put (49.5,25) {$0$}
		\put (57,27.5) {$\theta$}
		\put (80,27.5) {$S_\theta$}		
	\end{overpic}
	\caption{The comb domain $C$ and the sector $S_\theta$.}
	\label{cs}	
\end{figure}

Since $S_\theta \subset C$ we have ${\rm h} (S_\theta)\ge {\rm h} ( C)$. We know that ${\rm h} (S_\theta)= \pi/\theta$ and thus
\begin{equation}\label{an1}
{\rm h} ( C)\le \frac{\pi}{\theta}.
\end{equation}
Next, we prove the reverse inequality. Let  ${F_r } =C\cap \{|z|=r\}$ for $r>0$. Since we will estimate the Hardy number with the aid of Theorem \ref{hardy}, first we need to find a lower estimate for $d_C (0,F_r )$. For a fixed $r\ge r_0$,  where $r_0>1$ is an absolute constant that we define later, let $F_r^*$ be a component of $F_r$ for which $d_C (0,F_r )=d_C (0,F_r^* )$. By (\ref{hahy}) and Theorem \ref{upperest}, we have
\begin{align}\label{length}
d_C (0,F_r)&=d_C (0,F_r^* ) \ge \log \frac{2}{\pi} + \log \frac{1}{\omega _C (0,F_r^* )}  \nonumber\\
&\ge \log \frac{2}{\pi} + \log \frac{\pi}{8}+ \pi \int_{1}^r {\frac{{dt}}{{t\Phi_r (t)}}} =\log\frac{1}{4}+\pi\int_{1}^r {\frac{{dt}}{{t\Phi_r (t)}}},
\end{align}
where ${t\Phi_r (t)}$ is the length of the component of $C\cap \{|z|=t\}$ that separates $0$ from $F_r^*$. If $J_t$ is the component of $C\cap \{|z|=t\}$ that intersects the real axis, then we will show that that
$$t\Phi_r(t)\leq l(J_t)$$
for $t\in[r_0,r)$. Because of the symmetry of $C$ with respect to the real and the imaginary axis, it suffices to consider the following two cases for $F_r^*$. 

\vspace{1em}

\noindent
\textit{Case I:} Suppose that $F_r^*$ is the component of $F_r$ intersecting the positive real axis.
In this case, the component of $C\cap \{|z|=t\}$ that separates $0$ from $F_r^*$ is $J_t$ (see Fig. \ref{c1}). Therefore, ${t\Phi (t)}=l(J_t)$ for every $t\in [1,r)$. 

\vspace{1em}

\noindent
\textit{Case II:} Suppose that $F_r^*$ is any of the components of $F_r$ that lies in the first quadrant. In this case, there is some maximal point $r^*\geq 1$ (depending on $F_r^*$) such that for $t\in[1,r^*)$ the component of $C\cap \{|z|=t\}$ that separates $0$ from $F_r^*$ is $J_t$ and for $t\in [r^*,r)$ the component of $C\cap \{|z|=t\}$ that separates $0$ from $F_r^*$ is an arc $I_t^r$ whose endpoints lie on the same rays that contain the endpoints of $F_r^*$ (see Fig. \ref{c2}). Therefore,
\begin{equation}\label{angle}
t\Phi_r (t) = \left\{ \begin{array}{l}
l(J_t ),\,\,t\in[1,r^*) \\ 
l(I_t^r ),\,\,t\in [r^*,r) \\ 
\end{array} \right..
\end{equation}
Now, we prove the following lemma.

\begin{figure}
	
	\centering
	\begin{minipage}{0.52\textwidth}
		\centering 
			\begin{overpic}[width=\textwidth]{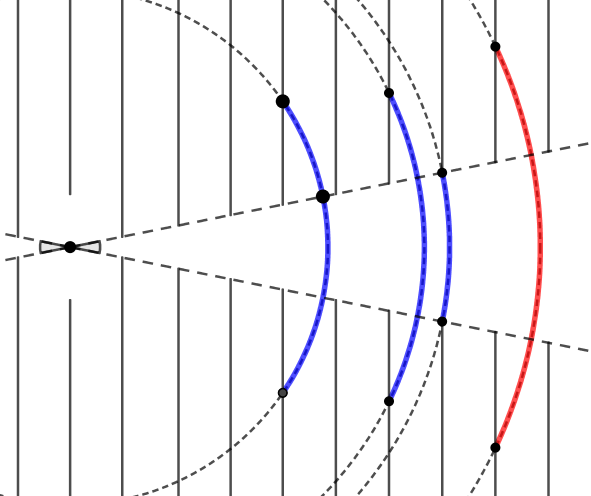}
				\put  (93,40) {\textcolor{red}{$F_r^*$}}
				\put (60,40) {\textcolor{blue}{$J_t$}}
				\put (49,46) {$v_t$}
				\put (41,65) {$s_t$}
				\put (25,40) {$\theta$}
				\put (11,36) {$0$}
			\end{overpic}
		\caption{\textit{Case I}.}
			\label{c1}
	
\end{minipage}
	\begin{minipage}{0.44\textwidth}
	\centering
		\begin{overpic}[width=\textwidth]{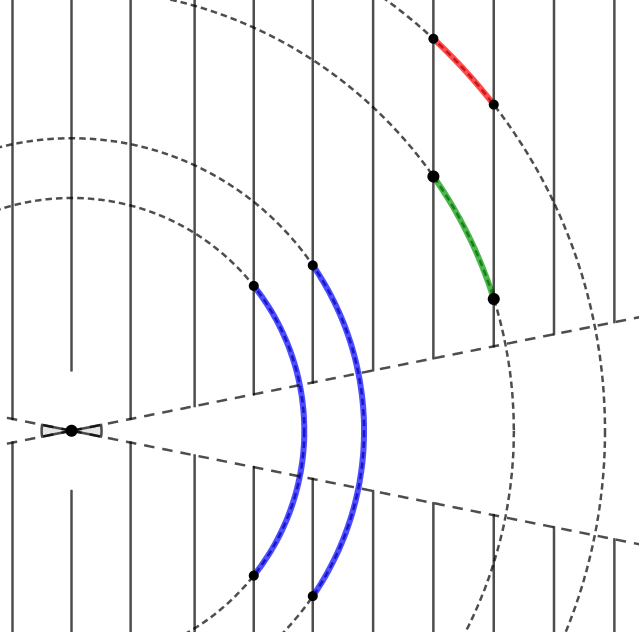}
			\put  (69,93) {\textcolor{red}{$F_r^*$}}
			\put (70,70) {\textcolor{ForestGreen}{$I_t^r$}}
			\put (60,30) {\textcolor{blue}{$J_t$}}
			\put (62,66) {$z_t$}
			\put (70,48) {$w_t$}
			\put (25,29) {$\theta$}
			\put (10,24) {$0$}
		\end{overpic}
		\caption{\textit{Case II}.}
		\label{c2}
	\end{minipage}	
\end{figure}

\begin{lemma}\label{lem1} There is an absolute constant $r_0>1$ such that, for every $t\in [r_0,r)$,
\[l(I_t^r )\le l(J_t ).\]	
\end{lemma}
\begin{proof}
Let $z_t$ and $w_t$ be the endpoints of $I_t^r$ (see Fig. \ref{c2}). In polar coordinates, 
\[z_t=te^{i\phi_1(t)}\,\,{\text{and}}\,\,w_t=te^{i\phi_2(t)},\]
where $\theta/2<\phi_2(t)<\phi_1(t)<\pi/2$. Since the distance between two consecutive rays of $C$ is always equal to $1$, we infer that
\[t\cos\phi_2(t)-t\cos\phi_1(t)=1.\]
If $\phi(t)=\phi_1(t)-\phi_2(t)$, the relation above in combination with the mean value theorem implies that
\[0<\phi(t) \le \frac{\cos\phi_2(t)-\cos\phi_1(t)}{c(\theta)}=\frac{1}{tc(\theta)},\]
where $c(\theta)= \sin(\theta/2)$. Hence, $\phi(t)\to 0$ as $t\to +\infty$. So, there is a constant $r_0>1$ such that $\phi(t)\le \theta$ for every $t\in [r_0,r)$. This implies that, for every $t\in [r_0,r)$, 
\[l(I_t^r)=t\phi (t)\le t\theta \le l(J_t)\]
and the proof is complete.
\end{proof}

By (\ref{angle}) and Lemma \ref{lem1}, we deduce that $t\Phi_r(t)\le l(J_t)$ for every $t\in [r_0,r)$. So, in both cases, for every $t\in [r_0,r)$,
\[t\Phi_r(t)\le l(J_t).\] 
If we set $l(J_t)=t\Theta (t)$, by (\ref{length}) it follows that, for every $r>r_0$,
\begin{equation}\label{hypint}
d_C (0,F_r)\ge \log\frac{1}{4}+\pi\int_{r_0}^r {\frac{{dt}}{{t\Theta (t)}}}.
\end{equation}
To complete the proof we need the following lemma.
\begin{lemma}\label{lem2} It is true that
\begin{equation}\nonumber
\lim_{t\to +\infty} \Theta (t)=\theta.
\end{equation}	
\end{lemma}
 \begin{proof} For $t>1$, if $s_t$ is the endpoint of the arc $J_t$ that lies in the first quadrant and $v_t$ is the intersection of $J_t$ with the line $y=\tan (\theta/2)x$ (see Fig. \ref{c1}), then
 	\[s_t=t\cos (\Theta(t)/2)+it\sin (\Theta(t)/2)\,\,{\text{and}}\,\,v_t=t\cos (\theta/2)+it\sin (\theta/2).\]
Since the distance between two consecutive  rays of $C$ is always equal to 1, we infer that
\[t\cos(\Theta(t)/2)+c(t)=t\cos(\theta/2),\]
where $c(t)\in(0,1]$. Therefore, $\cos(\Theta(t)/2) \to \cos(\theta/2)$ as $t\to+\infty$. Since $\theta/2\in (0,\pi/2)$ and $\Theta(t)/2\in [\theta/2,\pi/2)$ for every $t>1$, we have
\[\lim_{t\to +\infty} \Theta (t)=\theta\]
and the proof is complete.
\end{proof}
Fix an $\varepsilon  > 0$. By Lemma \ref{lem2} there is a $t_0>r_0$ (depending on $\varepsilon$) such that for every $t\ge t_0$,
\[|\Theta (t)-\theta|<\varepsilon\]
or
\[\frac{1}{\Theta (t)}>\frac{1}{\theta+\varepsilon}.\]
So, for $r>t_0$, we have
\begin{equation}\nonumber 
\int_{r_0}^r{\frac{{dt}}{{t\Theta (t)}}}\ge\int_{t_0}^r{\frac{{dt}}{{t(\theta+\varepsilon)}}}=\frac{1}{\theta+\varepsilon}\left( {\log r - \log t_0 } \right).
\end{equation}
Combining this with Theorem \ref{hardy} and (\ref{hypint}), we obtain
\begin{align}
{\rm h}(C)=\liminf_{r\to +\infty} \frac{d_C (0,F_r)}{\log r}\ge \pi \liminf_{r\to +\infty}\left(\frac{1}{\log r}\int_{r_0}^r{\frac{{dt}}{{t\Theta (t)}}}\right) \ge \frac{\pi}{\theta+\varepsilon}. \nonumber
\end{align} 
Therefore,
\[{\rm h}(C)\ge \frac{\pi}{\theta+\varepsilon}.\]
Letting $\varepsilon \to 0$ we have
\[{\rm h}(C)\ge \frac{\pi}{\theta}.\]
This in conjuction with (\ref{an1}) gives
\[{\rm h}(C)=\frac{\pi}{\theta}=p.\]
Hence, for every $p\in (1,+\infty)$, there is a comb domain $C$ with ${\rm h}(C)=p$.

\vspace{1em}

\noindent
{\bf Case 2:} Suppose that $p=1$. We consider the comb domain $C$ for which $x_n=n$ and $b_n=n^2+1$ for every $n\in \mathbb{Z}$ (see Fig. \ref{c3}). Fix an $\varepsilon \in (0,\pi)$. There is a sector $S_\varepsilon \subset C$ with angular opening equal to $\pi-\varepsilon$. Indeed, for $\alpha:=\tan((\pi-\varepsilon)/2)>0$ there is a natural number $n>(\alpha^2-4)/4\alpha$ for which the functions $y=x^2+1$ and $y=\alpha (x-n)$ (and their symmetric functions with respect to the real axis) have no common points. Thus, the sector $S_\varepsilon$ bounded by the ray $y=\alpha (x-n)$ for $x\ge n$ and its symmetric with respect to the real axis is contained in $C$ (see Fig. \ref{c3}). Moreover, the Hardy number of $S_\varepsilon$ is equal to $\pi/(\pi-\varepsilon)$. Since $S_\varepsilon \subset C$, we have ${\rm h}(C)\le {\rm h}(S_\varepsilon)$, or equivalently,
\[{\rm h}(C)\le \frac{\pi}{\pi-\varepsilon}.\]
Letting $\varepsilon \to 0$, we obtain ${\rm h}(C) \le1$. Recall that ${\rm h}(C)\ge 1$ (which holds for all comb domains). Hence, we have ${\rm h}(C)=1$.

\begin{figure}
	\begin{overpic}[width=\linewidth]{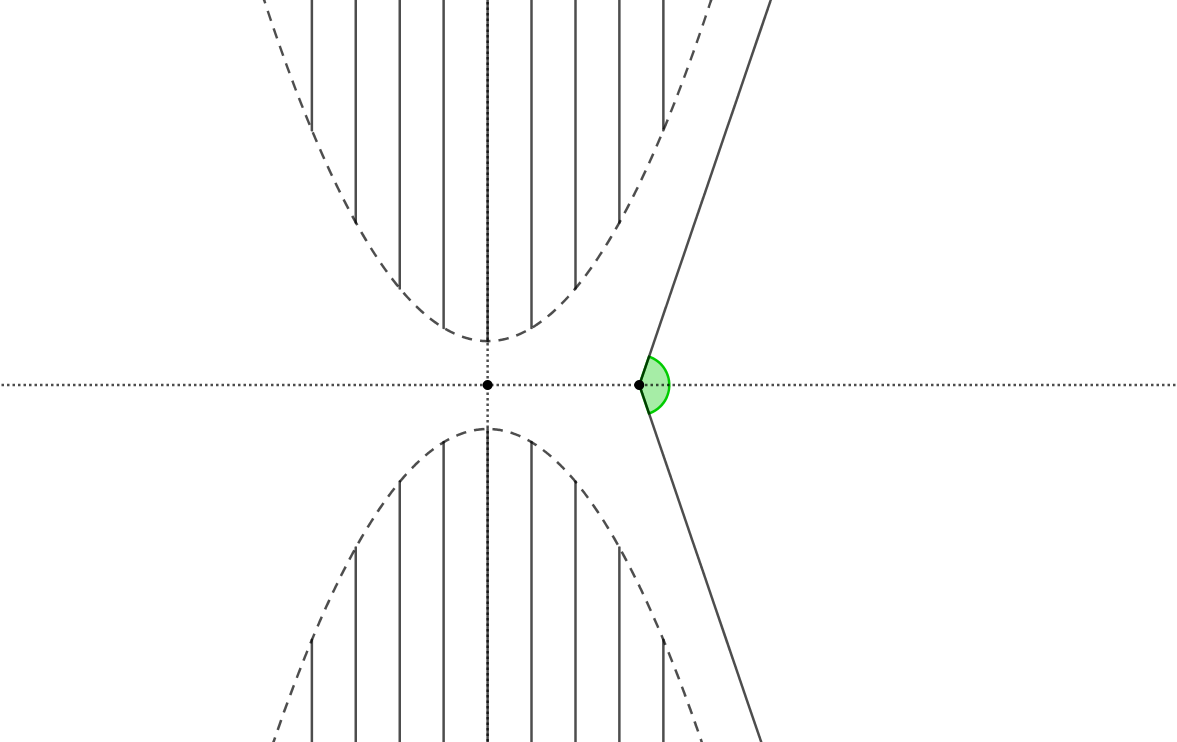}
		\put (39.5,28) {$0$}
		\put (52,28) {$n$}
		\put (70,37) {$S_\varepsilon$}
		\put (57,32) {$\pi-\varepsilon$}
		\put (64,55) {$y=\alpha(x-n)$}
		\put (63,8) {$y=-\alpha(x-n)$}		
	\end{overpic}
	\caption{The comb domain $C$ and the sector $S_\varepsilon$.}
	\label{c3}	
\end{figure}
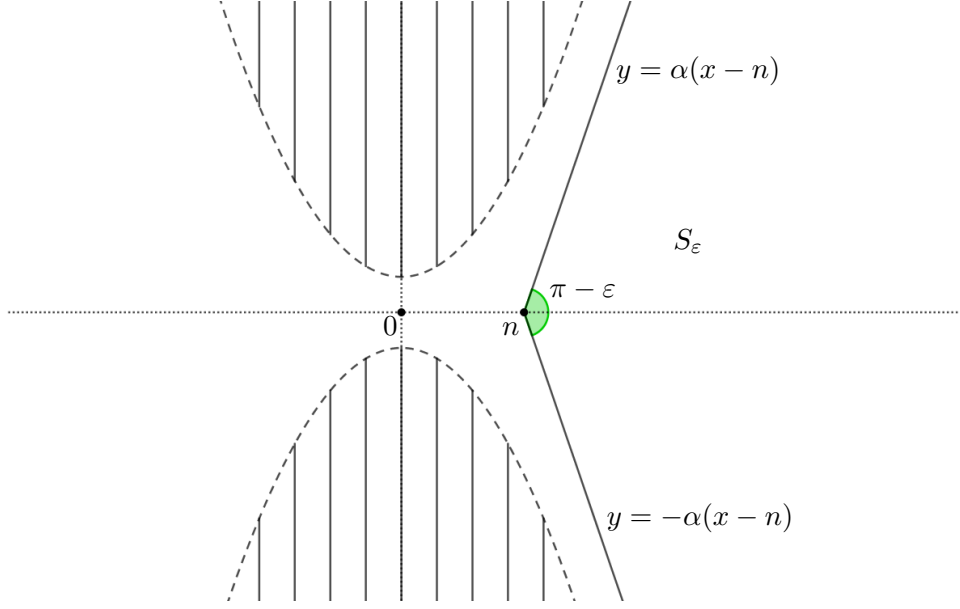 

\vspace{1em}

\noindent 
{\textit {Remark:}} We note that Theorem \ref{main} does not answer entirely Question \ref{que1} in case $q=1/2$ and $p<q$. However, the domain $C$ constructed here covers this case too. The Riemann mapping of $C$ is sector-like (see \cite[pp. 38--39]{Co1} for the definition and properties) and thus it does not belong to $H^1 (\mathbb{D})$. However, it belongs to $H^{2p} (\mathbb{D})$ for any $p<1/2$, because the Hardy number of comb domains is always greater than or equal to $1$. This means that $C$ has finite $p$-th moment but infinite $1/2$-th moment.

\vspace{1em}

\noindent
{\bf Case 3:} Suppose that $p=+\infty$. The domain $C$ with $x_n=n$ and $b_n=1$ for every $n\in\mathbb{Z}$ has Hardy number equal to infinity. This has already been proved in \cite{Kara}.
\qed

\end{document}